\documentclass[12pt]{article}
\usepackage{mathrsfs}
\usepackage{amsfonts,amsmath,amsthm, amssymb}
\usepackage{latexsym, euscript, epic, eepic}
\usepackage{lineno}

\newtheorem{cro}{Corollary}[section]
\newtheorem{defn}{Definition}[section]

\newtheorem{thm}{Theorem}[section]
\newtheorem{lem}{Lemma}[section]

\numberwithin{equation}{section}

\begin{document}

\title{Generic Points in Some Nonuniformly Hyperbolic Systems Via Pesin Theory
 \footnotetext {* Corresponding author}
  }
\author{Zheng Yin$^{1}$, Ercai Chen$^{*1,2}$, Xiaoyao Zhou$^{1}$ \\
  \small   1 School of Mathematical Sciences and Institute of Mathematics, Nanjing Normal University,\\
   \small   Nanjing 210023, Jiangsu, P.R.China\\
    \small 2 Center of Nonlinear Science, Nanjing University,\\
     \small   Nanjing 210093, Jiangsu, P.R.China.\\
      \small    e-mail: zhengyinmail@126.com ecchen@njnu.edu.cn zhouxiaoyaodeyouxian@126.com
}
\date{}
\maketitle

\begin{center}
 \begin{minipage}{120mm}
{\small {\bf Abstract.} This article is devoted to the investigation
of the topological pressure of generic points for nonuniformly
hyperbolic systems via Pesin theory. In particular, our result can
be applied to the nonuniformly hyperbolic diffeomorphisms described by Katok
and several other classes of diffeomorphisms derived from Anosov systems. }
\end{minipage}
 \end{center}

\vskip0.5cm
\section{Introduction}
We say $(M,d,f)$ is a topological dynamical system means that
$(M,d)$ is a compact metric space and $f:M\to M$ is a continuous
map. Let $\mathscr{M}(M)$, $\mathscr M_{\rm inv}(M,f)$ and $\mathscr M_{\rm erg}(M,f)$ be
the set of all Borel probability measures, $f$-invariant probability measures and ergodic measures
respectively. For an $f$-invariant subset $Z\subset X,$ let
$\mathscr M_{\rm inv}(Z,f)$ denote the subset of $\mathscr M_{\rm inv}(M,f)$
for which the measures $\mu$ satisfy $\mu(Z)=1$ and $\mathscr M_{\rm erg}(Z,f)$
denote those which are ergodic. Denote by $C^0(M)$ the space of continuous
functions from $M$ to $\mathbb{R}$ with the sup norm. For $\varphi\in C^0(M)$
and $n\geq1$ we denote $\sum_{i=0}^{n-1}\varphi(f^ix)$ by $S_n\varphi(x)$.
For every $\epsilon>0$, $n\in \mathbb{N}$ and a point $x\in M$, define
$B_n(x,\epsilon)=\{y\in M:d(f^ix,f^iy)<\epsilon,\forall 0\leq i\leq n-1\}$.
Given $x\in X,$ the $n$-ordered empirical measure of $x$ is
given by
\begin{align*}
\mathscr{E}_n(x)=\frac{1}{n}\sum\limits_{i=0}^{n-1}\delta_{f^ix},
\end{align*}
where $\delta_y$ is the Dirac mass at $y$. Denote by $V(x)$ the set
of limit measures of the sequence of measures $\mathscr{E}_n(x).$
Then $V(x)$ is always a compact connected subset of $\mathscr M_{\rm
inv}(Z,f).$ Let $G_\mu:=\left\{x\in X: V(x)=\{\mu\}\right\}$ be the
set of  generic points of $\mu.$ By Birkhoff ergodic theorem and
ergodic decomposition theorem, we have
\begin{itemize}
  \item $\mu(G_\mu)=1,$ if $\mu$ is ergodic,
  \item $\mu(G_\mu)=0,$ if $\mu$ is non-ergodic.
\end{itemize}
In 1973, Bowen \cite{Bow} proved that the metric entropy of $\mu$ is
equal to the topological entropy of $G_\mu$ i.e., $h_{\rm top}
(G_\mu,f)=h_\mu(f)$ when $\mu$ is ergodic.  In 2007, Pfister \&
Sullivan \cite{PfiSul} showed that for any $f$-invariant measure
$\mu,$
\begin{align} \label{f1.1}
 h_{\rm top} (G_\mu,f)=h_\mu(f),
\end{align}
when the topological dynamical system $(X,d,f)$ is endowed with
$g$-almost product property (a weaker form of specification). This
implies the conditional principle of Takens and Verbitskiy
\cite{TakVer}
\begin{equation}\begin{split}\label{f1.2}
 & h_{\rm top} \left(\left\{x\in X:\lim\limits_{n\to\infty}\frac{1}{n}\sum\limits_{i=0}^{n-1}\varphi
 (f^ix) = \alpha\right\},f\right)  \\= & \sup\left\{h_\mu(f):\mu\in \mathscr M_{\rm inv}(X,f),\int \varphi d\mu=\alpha\right\},
\end{split}\end{equation} where $\varphi$ is a continuous function and
$\alpha\in\mathbb{R}.$ In the view of large deviations theory,
(\ref{f1.1}) can be seen as level-2 and (\ref{f1.2}) can be seen as
level-1. Recently, (\ref{f1.1}) and (\ref{f1.2}) are extended to
topological pressure by Pei \& Chen \cite{PeiChe} and Thompson
\cite{Tho1}, respectively. See Yamamoto \cite{Yam1} for higher
version and  Feng \& Huang \cite{FenHua} for
sub-additive case.

The investigation of multifractal analysis for nonuniformly hyperbolic systems
has attracted more and more attentions. See
Bomfim \& Varandas \cite{BomVar},  Chung \& Takahasi \cite{ChuTak}
and  Thompson \cite{Tho}. Very recently, Liang et al.
\cite{LiaLiaSUnTia} showed that for an ergodic hyperbolic measure
$\omega$ of a $C^{1+\alpha}$ diffeomorphism $f$ on a Riemannian
manifold $X,$ there is an $\omega$-full measured set
$\widetilde{\Lambda}$ such that for every invariant probability
measure $\mu\in \mathscr M_{\rm inv}(\widetilde{\Lambda},f),$
\begin{align*}
 h_{\rm top} (G_\mu,f)=h_\mu(f).
\end{align*}
The purpose of this article is extending the above result to
topological pressure. We use the Generalised Pressure Distribution Principle
(see lemma \ref{lem3.5}) to estimate the lower bound of the
topological of $G_\mu$, which is different from \cite{LiaLiaSUnTia}.

This article is organized as follows. In section 2, we provide some notions
and results of Pesin theory and state the main results. Section 3 is devoted
to the proof of the main result. Examples and applications are given in section 4.


\section{ Preliminaries}
In this section, we first present some some notions and results of
Pesin theory \cite{BarPes2,KatHas,Pol}. Then we introduce the definition
of topological pressure and state the main results.

Suppose $M$ is a compact connected boundary-less Riemannian
$n$-dimension manifold and $f:X\to X$ is a $C^{1+\alpha}$
diffeomorphism. Let $\mu\in\mathscr M_{\rm erg}(Z,f)$ and $Df_x$ denote
the tangent map of $f$ at $x\in M.$ We say that
$x\in X$ is a regular point of $f$ if there exist
$\lambda_1(\mu)>\lambda_2(\mu)>\cdots>\lambda_{\phi(\mu)}(\mu)$ and a
decomposition on the tangent space $T_x
M=E_1(x)\oplus\cdots\oplus E_{\phi(\mu)}(x)$ such that
\begin{align*}
\lim\limits_{n\to\infty}\frac{1}{n}\log\|(Df^n_x)u\|=\lambda_j(x),
\end{align*}
where $0\neq u\in E_j(x), 1\leq j\leq \phi(\mu).$ The number $\lambda_j(x)$ and
the space $E_j(x)$ are called the Lyapunov exponents and the eigenspaces of
$f$ at the regular point $x,$ respectively. Oseledets theorem \cite{Ose} say that all
regular points forms a Borel set with total measure. For a regular point
$x\in M$, we define
\begin{align*}
\lambda^+(\mu)=\min\{\lambda_i(\mu)|\lambda_i(\mu)\geq0,1\leq i\leq \phi(\mu)\}
\end{align*}
and
\begin{align*}
\lambda^-(\mu)=\min\{-\lambda_i(\mu)|\lambda_i(\mu)\leq0,1\leq i\leq \phi(\mu)\}.
\end{align*}
We appoint $\min\emptyset=0$. An ergodic measure $\mu$ is hyperbolic if $\lambda^+(\mu)$
and $\lambda^-(\mu)$ are both non-zero.

\begin{defn}
Given $\beta_1,\beta_2\gg\epsilon>0$ and for all $k\in\mathbb{Z}^+,$
the hyperbolic block $\Lambda_k=\Lambda_k(\beta_1,\beta_2,\epsilon)$
consists of all points $x\in M$  such that there exists a
decomposition $T_xM=E_x^s\oplus E_x^u$ with invariance property
$Df^t(E_x^s)=E^s_{f^tx}$ and $Df^t(E_x^u)=E^u_{f^tx}$, and satisfying:
\begin{itemize}
  \item $\|Df^n|E^s_{f^tx}\|\leq e^{\epsilon k}e^{-(\beta_1-\epsilon)n}e^{\epsilon|t|},\forall t\in\mathbb{Z},n\geq1;$
  \item $\|Df^{-n}|E^u_{f^tx}\|\leq e^{\epsilon k}e^{-(\beta_2-\epsilon)n}e^{\epsilon|t|},\forall t\in\mathbb{Z},n\geq1;$
  \item $\tan (\angle(E^s_{f^tx},E^u_{f^tx}))\geq e^{-\epsilon k}e^{-\epsilon|t|},\forall t\in\mathbb{Z}.$
\end{itemize}
\end{defn}

\begin{defn}
$
\Lambda(\beta_1,\beta_2,\epsilon)=\bigcup\limits_{k=1}^\infty\Lambda_k(\beta_1,\beta_2,\epsilon)
$ is a Pesin set.
\end{defn}
The following statements are elementary properties of Pesin blocks (see \cite{Pol}):
\begin{itemize}
  \item[(1)] $\Lambda_1\subseteq\Lambda_2\subseteq\cdots;$
  \item[(2)] $f(\Lambda_k)\subseteq\Lambda_{k+1},f^{-1}(\Lambda_k)\subseteq\Lambda_{k+1};$
  \item[(3)] $\Lambda_k$ is compact for each $k\geq1$;
  \item[(4)] For each $k\geq1$, the splitting $\Lambda_k\ni x\mapsto E_x^s\oplus E_x^u$ is continuous.
\end{itemize}
The Pesin set $\Lambda(\beta_1,\beta_2,\epsilon)$ is an
$f$-invariant set but usually not compact. Given an ergodic measure
$\mu\in\mathscr M_{\rm erg}(M,f)$, denote by $\mu|\Lambda_l$ the
conditional measure of $\mu$ on $\Lambda_l.$ Let
$\widetilde{\Lambda}_l=$ supp$(\mu|\Lambda_l)$ and
$\widetilde{\Lambda}_\mu=\bigcup_{l\geq1}\widetilde{\Lambda}_l.$
If $\omega$ is an ergodic hyperbolic measure for $f$ and $\beta_1\leq\lambda^-(\omega)$ and
$\beta_2\leq\lambda^+(\omega)$, then $\omega\in \mathscr M_{\rm inv}(\widetilde{\Lambda}_\omega,f)$.

\textbf{Lyapunov metric.}
Suppose $\Lambda(\beta_1,\beta_2,\epsilon)=\bigcup_{l\geq1}\Lambda_k(\beta_1,\beta_2,\epsilon)$
is a nonempty Pesin set. Let $\beta'_1=\beta_1-2\epsilon,
\beta'_2=\beta_2-2\epsilon.$ It follows from $\epsilon\ll
\beta_1,\beta_2$ that $\beta'_1>0,\beta'_2>0.$ Given $x\in
\Lambda(\beta_1,\beta_2,\epsilon),$ we define
\begin{align*}
\|v_s\|_s&=\sum\limits_{n=1}^\infty e^{\beta'_1n}\|Df^n_x(v_s)\|,
\forall v_s\in E_x^s,\\
\|v_u\|_u&=\sum\limits_{n=1}^\infty e^{\beta'_2n}\|Df^{-n}_x(v_u)\|,
\forall v_u\in E_x^u,\\
\|v\|'&=\max(\|v_s\|_s,\|v_u\|_u), {\rm where~} v=v_s+v_u.
\end{align*}
The norm $\|\cdot\|'$ is called Lyapunov metric, which is not
equivalent to the Riemannian metric generally.
With the Lyapunov metric $f:\Lambda\to\Lambda$
is uniformly hyperbolic. The following estimates are known:
\begin{enumerate}
\item[(i)] $\|Df|E_x^s\|'\leq e^{-\beta_1'}$, $\|Df^{-1}|E_x^u\|'\leq e^{-\beta_2'}$;
\item[(ii)] $\frac{1}{\sqrt n}\|v\|_x\leq \|v\|_x^\prime
\leq\frac{2}{1-e^{-\epsilon}} e^{\epsilon k}\|v\|_x$, $\forall v\in T_xM,x\in \Lambda_k$.
\end{enumerate}

\textbf{Lyapunov neighborhood.}
Fix a point $x\in \Lambda(\beta_1,\beta_2;\epsilon)$
By taking charts about $x$ and $f(x),$  we can assume
without loss of generality that $x\in\mathbb{R}^d,
f(x)\in\mathbb{R}^d.$ For a sufficiently small neighborhood $U$ of
$x,$ we can trivialize tangent bundle over $U$ by identifying $T_U M\equiv
U\times\mathbb{R}^d.$ For any point $y\in U$ and tangent vector
$v\in T_yM,$ we can then use the identification
$T_U M\equiv U\times\mathbb{R}^d$ to
translate the vector $v$ to a corresponding
vector $\overline{v}\in T_xM.$ We then define
$\|v\|''_y=\|\overline{v}\|'_x$, where $\|\cdot\|''$ indicates
the Lyapunov metric. This define a new norm $\|\cdot\|''$(which agrees
with $\|\cdot\|'$ on the fiber $T_xM$). Similarly, we can define
$\|\cdot\|''_z$ on $T_zM$ (for any $z$ in a sufficiently small
neighborhood of $fx$ or $f^{-1}x$). We write $\overline{v}$ as $v$
whenever there is no confusion. We can define a new splitting
$T_yM=E_y^{s'}\oplus E_y^{u'},y\in U$ by translating the splitting
$T_xM=E_x^s\oplus E_x^u$ (and similarly for $T_zM=E_z^{s'}\oplus E_z^{u'}$).

There exist
$\beta''_1=\beta_1-3\epsilon>0,\beta''_2=\beta_2-3\epsilon>0$ and
$\epsilon_0>0$ such that if we set
$\epsilon_k=\epsilon_0\exp(-\epsilon k)$ then for any $y\in
B(x,\epsilon_k)$ in an $\epsilon_k$ neighborhood of $x\in\Lambda_k,$
we have a splitting $T_y M=E_y^{s'}\oplus E_y^{u'}$ with hyperbolic behaviour:
\begin{enumerate}
\item[(i)] $\|Df_y(v)\|''_{fy}\leq e^{-\beta''_1}\|v\|''$ for every $v\in
E_y^{s'};$
\item[(ii)] $\|Df^{-1}_y(w)\|''_{f^{-1}y}\leq e^{-\beta''_2}\|w\|''$ for
every $w\in E_y^{u'}.$
\end{enumerate}

\begin{defn}
We define the Lyapunov neighborhood $\prod=\prod(x,a\epsilon_k)$ of
$x\in\Lambda_k$(with size $a\epsilon_k,0<a<1$)
to be the neighborhood of $x$ in $X$
which is the exponential projection onto $M$ of the tangent
rectangle $(-a\epsilon_k,a\epsilon_k)E_x^s\oplus
(-a\epsilon_k,a\epsilon_k)E_x^u.$
\end{defn}

Let $\{\delta_k\}_{k=1}^\infty $ be a sequence of positive real
numbers. Let $\{x_n\}_{n=-\infty}^\infty$ be a sequence of points in
$\Lambda=\Lambda(\beta_1,\beta_2,\epsilon)$ for which there exists a
sequence $\{s_n\}_{n=-\infty}^{\infty}$ of positive integers satisfying:
\begin{equation*}\begin{split}
&\text{(a) } x_n\in\Lambda_{s_n},\forall n\in\mathbb{Z};\\
&\text{(b) } |s_n-s_{n-1}|\leq 1, \forall n\in\mathbb{Z};\\
&\text{(c) } d(f(x_n),x_{n+1})\leq\delta_{s_n},  \forall n\in\mathbb{Z},
\end{split}\end{equation*}
then we call $\{x_n\}_{n=-\infty}^\infty$ a $\{\delta_k\}_{k=1}^\infty$ pseudo-orbit.  Given $\eta>0$
a point $x\in M$ is an $\eta$-shadowing point for the
$\{\delta_k\}_{k=1}^\infty$ pseudo-orbit  if $d(f^n(x),x_n)\leq
\eta\epsilon_{s_n},\forall n\in\mathbb{Z},$ where
$\epsilon_k=\epsilon_0e^{-\epsilon k}$ are given by the definition of Lyapunov neighborhoods.

\vskip0.3cm

\noindent \textbf{Weak shadowing lemma.} \cite{Hir,KatHas,Pol}
{\it
Let $f:M\to M$ be a $C^{1+\alpha}$ diffeomorphism, with a non-empty
Pesin set $\Lambda=\Lambda(\beta_1,\beta_2,\epsilon)$ and fixed
parameters, $\beta_1,\beta_2\gg \epsilon>0.$ For $\eta>0$ there exists
a sequence $\{\delta_k\}$ such that for any $\{\delta_k\}$ pseudo-orbit
there exists a unique $\eta$-shadowing point.
}
\vskip0.3cm

\begin{defn}{\rm\cite{Pes}}
Suppose $Z\subset X$ be an arbitrary Borel set and $\psi\in C(X)$.
Let $\Gamma_{n}(Z,\epsilon)$ be the
collection of all finite or countable covers of $Z$ by sets of the
form $B_{m}(x,\epsilon),$ with $m\geq n$. Let $S_{n}\psi(x):=\sum_{i=0}^{n-1}\psi(T^{i}x)$. Set
\begin{align*}
M(Z,t,\psi,n,\epsilon):=\inf_{\mathcal{C}\in\Gamma_{n}(Z,\epsilon)}\left\{\sum_{B_{m}(x,\epsilon)\in
\mathcal{C}}\exp (-tm+\sup_{y\in
B_{m}(x,\epsilon)}S_{m}\psi(y))\right\},
\end{align*}
and
\begin{align*}
M(Z,t,\psi,\epsilon)=\lim_{n\to\infty}M(Z,t,\psi,n,\epsilon).
\end{align*}
Then there exists a unique number $P(Z,\psi,\epsilon)$ such that
$$P(Z,\psi,\epsilon)=\inf\{t:M(Z,t,\psi,\epsilon)=0\}=\sup\{t:M(Z,t,\psi,\epsilon)=\infty\}.$$
$P(Z,\psi)=\lim_{\epsilon\to0}P(Z,\psi,\epsilon)$ is called
the topological pressure of $Z$ with respect to $\psi$.
\end{defn}

It is  obvious that the following hold:
\begin{enumerate}
  \item[(1)] $P(Z_1,\psi)\leq P(Z_2,\psi)$ for any $Z_1\subset Z_2\subset X$;
  \item[(2)] $P(Z,\psi)=\sup_i P(Z_i,\psi)$, where $Z=\bigcup_i Z_i\subset X$.
\end{enumerate}

Now, we state the main results of this paper as follows:
\begin{thm}\label{thm2.2}
Let $f:M\to M$ be a $C^{1+\alpha}$ diffeomorphism of a compact Riemannian manifold, with a non-empty
Pesin set $\Lambda=\Lambda(\beta_1,\beta_2,\epsilon)$ and fixed
parameters, $\beta_1,\beta_2\gg \epsilon>0$ and
let $\mu\in \mathscr{M}_{erg}(M,f)
$ be any ergodic measure. Then for every  $\nu\in \mathscr M_{\rm inv}(\widetilde{\Lambda}_{\mu},f),
\psi\in C^0(M),$ we have
\begin{align*}
P(G_\nu,\psi)=h_\nu(f)+\int \psi d\nu.
\end{align*}
\end{thm}
\begin{cro}
Let $f:M\to M$ be a $C^{1+\alpha}$ diffeomorphism of a compact Riemannian manifold
and let $\omega\in\mathscr M_{\rm erg}(M,f)$ be a hyperbolic measure. If
$\beta_1\leq\lambda^-(\omega)$ and $\beta_2\leq\lambda^+(\omega)$,
then for every  $\nu\in \mathscr M_{\rm inv}(\widetilde{\Lambda}_{\omega},f),
\psi\in C^0(M),$ we have
\begin{align*}
P(G_\nu,\psi)=h_\nu(f)+\int \psi d\nu.
\end{align*}
\end{cro}
\section{Proof of Main Theorem}
In this section, we will verify theorem \ref{thm2.2}.
The upper bound on $P(G_\nu,\psi)$ is easy to get.
To obtain the lower bound estimate we need to construct
a suitable pseudo-orbit and a sequence of measures
to apply Generalised Pressure Distribution Principle.
Our method is inspired by \cite{LiaLiaSUnTia},
\cite{PfiSul} and \cite{Tho2}. The proof will be
divided into the following two subsections.

\subsection{Upper Bound on $P(G_\nu,\psi)$}
The upper bound of  $P(G_\nu,\psi)$ holds without extra assumption.
By \cite{PeiChe,ZhoChe}, we have
\begin{align*}
P(G_\nu,\psi)\leq h_\nu(f)+\int \psi d\nu.
\end{align*}

\subsection{Lower Bound on $P(G_\nu,\psi)$}
The aim of this section is to obtain the lower bound of $P(G_\nu,\psi)$.
Our tool is Generalised Pressure Distribution Principle.

\subsubsection{Katok's Definition of Measure-theoretic Pressure}
For $\nu\in\mathscr{M}_{inv}(M,f)$ and $\psi\in C^0(M)$, the
measure-theoretic pressure of $f$ respect to $\psi$ and $\nu$ is
\begin{align*}
P_{\nu}(f,\psi):=h_{\nu}(f)+\int\psi d\nu.
\end{align*}
We use the Katok's definition of measure-theoretic pressure based on the following
lemma.
\begin{lem}{\rm\cite{Men}}\label{lem3.1}
Let $(M,d)$ be a compact metric space, $f:M\to M$ be a continuous map and $\mu$
be an ergodic invariant measure. For $\epsilon>0$, $\delta\in (0,1)$ and $\psi\in C^0(M)$, define
\begin{align*}
N^\mu(\psi,\delta,\epsilon,n):=\inf\left\{\sum_{x\in S}\exp\left\{\sum_{i=0}^{n-1}\psi(f^ix)\right\}\right\},
\end{align*}
where the infimum is taken over all sets $S$ which $(n,\epsilon)$ span some set $Z$ with $\mu(Z)>1-\delta$.
We have
\begin{align*}
h_\mu(f)+\int\psi d\mu=\lim_{\epsilon\to0}\liminf_{n\to\infty}\frac{1}{n}\log N^\mu(\psi,\delta,\epsilon,n).
\end{align*}
The formula remains true if we replace the $\liminf$ by $\limsup$.
\end{lem}

For $\epsilon>0$ and $\nu\in\mathscr{M}_{erg}(M,f)$, we define
\begin{align*}
P_\nu^{Kat}(f,\psi,\epsilon):=\liminf_{n\to\infty}\frac{1}{n}\log N^\nu(\psi,\delta,\epsilon,n).
\end{align*}
Then by lemma \ref{lem3.1},
\begin{align*}
P_\nu(f,\psi)=\lim_{\epsilon\to0}P_\nu^{Kat}(f,\psi,\epsilon).
\end{align*}
If $\nu$ is non-ergodic, we will define $P_\nu^{Kat}(f,\psi,\epsilon)$ by the ergodic
decomposition of $\nu$. The following lemma is necessary.

\begin{lem}\label{lem3.2}
Fix $\epsilon,\delta>0$ and $n\in\mathbb{N}$, the function $s:\mathscr{M}_{erg}(M,f)\to\mathbb{R}$
defined by $\nu\mapsto N^\nu(\psi,\delta,\epsilon,n)$ is upper semi-continuous.
\begin{proof}
Let $\nu_k\to\nu$. Let $a>N^\nu(\psi,\delta,\epsilon,n)$, then there exists a set $S$
which $(n,\epsilon)$ span some set $Z$ with $\nu(Z)>1-\delta$ such that
\begin{align*}
a>\sum_{x\in S}\exp\left\{\sum_{i=0}^{n-1}\psi(f^ix)\right\}.
\end{align*}
If $k$ is large enough, then $\nu_k(\bigcup_{x\in S}B_n(x,\epsilon))>1-\delta$, which implies that
\begin{align*}
a>N^{\nu_{k}}(\psi,\delta,\epsilon,n).
\end{align*}
Thus we obtain
\begin{align*}
N^\nu(\psi,\delta,\epsilon,n)\geq\limsup_{k\to\infty}N^{\nu_{k}}(\psi,\delta,\epsilon,n),
\end{align*}
which completes the proof.
\end{proof}
\end{lem}

Lemma \ref{lem3.2} tells us that the function $\overline{s}:\mathscr{M}_{erg}(M,f)\to\mathbb{R}$
defined by $\overline{s}(m)=P_m^{Kat}(f,\psi,\epsilon)$ is measurable.
Assume $\nu=\int_{\mathscr{M}_{erg}(M,f)}md\tau(m)$ is the ergodic decomposition of $\nu$.
Define
\begin{align*}
P_\nu^{Kat}(f,\psi,\epsilon):=\int_{\mathscr{M}_{erg}(M,f)}P_m^{Kat}(f,\psi,\epsilon)d\tau(m).
\end{align*}
We remark that for all $\nu\in \mathscr{M}_{inv}(M,f)$,
\begin{align*}
-\|\psi\|\leq P_{\nu}^{Kat}(f,\psi,\epsilon)\leq h_{top}(f)+\|\psi\|,
\end{align*}
where $h_{top}(f)$ is the topological entropy of $f$. Thus there exists $H>0$ such that
\begin{align*}
|P_{\nu}^{Kat}(f,\psi,\epsilon)|\leq H.
\end{align*}
By dominated convergence theorem, we have
\begin{align}\label{equ3.1}
P_\nu(f,\psi)=\int_{\mathscr{M}_{erg}(M,f)}\lim_{\epsilon\to0}P_m^{Kat}(f,\psi,\epsilon)d\tau(m)
=\lim_{\epsilon\to0}P_\nu^{Kat}(f,\psi,\epsilon).
\end{align}


\subsubsection{Some Lemmas}
For $\mu,\nu\in \mathscr{M}(M),$ define a compatible metric $D$ on
$\mathscr{M}(M)$ as follows:
\begin{align*}
D(\mu,\nu):=\sum\limits_{i\geq1}\frac{|\int\varphi_id\mu-\int\varphi_id\nu|}{2^{i+1}\|\varphi_i\|}
\end{align*}
where $\{\varphi_i\}_{i=1}^\infty$ is the dense subset of
$C^0(M)$. It is obvious that $D(\mu,\nu)\leq1$ for any $\mu,\nu\in \mathscr{M}(M).$
For any integer $k\geq1$ and $\varphi_1,\cdots,\varphi_k,$ there exists
$b_k>0$ such that
\begin{align*}
d(\varphi_j(x),\varphi_j(y))<\frac{1}{k}\|\varphi_j\|
\end{align*}
for any $d(x,y)<b_k, 1\leq j\leq k.$

\begin{lem}\label{lem3.3}
For any integer $k\geq1$ and invariant measure
$\nu\in\mathscr{M}_{inv}(\widetilde{\Lambda}_{\mu},f)$,
there exists a finite convex combination of ergodic probability measures with
rational coefficients $\mu_k=\sum\limits_{j=1}^{s_k}a_{k,j}m_{k,j}$ such that
\begin{align*}
D(\nu,\mu_k)\leq\frac{1}{k}, m_{k,j}(\widetilde{\Lambda}_{\mu})=1, ~{\rm and} ~
P_{\nu}^{Kat}(f,\psi,\epsilon)\leq\sum_{j=1}^{s_{k}}a_{k,j}P_{m_{k,j}}^{Kat}(f,\psi,\epsilon).
\end{align*}
\begin{proof}
Let
\begin{align*}
\nu=\int_{\mathscr{M}_{erg}(\widetilde{\Lambda}_{\mu},f)}md\tau(m)
\end{align*}
be the ergodic decomposition of $\nu$. Choose $N$ large enough such that
\begin{align*}
\sum\limits_{n=N+1}^{\infty}\frac{|\int\varphi_nd\mu-\int\varphi_nd\nu|}{2^{n+1}\|\varphi_n\|}<\frac{1}{3k}.
\end{align*}
We may assume that $\varphi_n\neq0$ for $n=1,\cdots,N$.
We choose $\zeta>0$ such that $D(\nu_1,\nu_2)<\zeta$ implies that
\begin{align*}
\left|\int \varphi_n d\nu_1-\int \varphi_n d\nu_2\right|<\frac{\|\varphi_n\|}{3k},n=1,2,\cdots,N.
\end{align*}
Let $\{A_{k,1},A_{k,2},\cdots,A_{k,s_k}\}$ be a partition of $\mathscr{M}_{erg}(\widetilde{\Lambda}_{\mu},f)$
with diameter smaller than $\zeta$. For any $A_{k,j}$ there exists an ergodic $m_{k,j}\in A_{k,j}$ such that
\begin{align*}
\int_{A_{k,j}}P_m^{Kat}(f,\psi,\epsilon)d\tau(m)
\leq\tau(A_{k,j})P_{m_{k,j}}^{Kat}(f,\psi,\epsilon).
\end{align*}
Obviously $m_{k,j}(\widetilde{\Lambda}_{\mu})=1$ and
$P_{\nu}^{Kat}(f,\psi,\epsilon)\leq\sum_{j=1}^{s_{k}}\tau(A_{k,j})P_{m_{k,j}}^{Kat}(f,\psi,\epsilon)$.
Let us choose rational numbers $a_{k,j}>0$
such that
$$|a_{k,j}-\tau(A_{k,j})|<\frac{1}{3ks_k}$$
and
$$P_{\nu}^{Kat}(f,\psi,\epsilon)\leq\sum_{j=1}^{s_{k}}a_{k,j}P_{m_{k,j}}^{Kat}(f,\psi,\epsilon).$$
Let
\begin{align*}
\mu_k=\sum\limits_{j=1}^{s_k}a_{k,j}m_{k,j}.
\end{align*}
By ergodic decomposition theorem, one can readily verify that
\begin{align*}
\left|\int\varphi_n d\nu-\int \varphi_n d\mu_k\right|\leq\frac{2\|\varphi_n\|}{3k},n=1,\cdots,N.
\end{align*}
Thus, we obtain
\begin{align*}
D(\nu,\mu_k)\leq\frac{1}{k}.
\end{align*}
\end{proof}
\end{lem}

\begin{lem}{\rm \cite{Boc}}\label{lem3.4}
Let $f:M\to M$ be a $C^{1}$ diffeomorphism of a compact Riemannian manifold and $\mu\in\mathscr M_{\rm inv}(M,f)$.
Let $\Gamma\subseteq M$ be a measurable set with $\mu(\Gamma)>0$ and let
\begin{align*}
\Omega=\bigcup_{n\in \mathbb{Z}}f^n(\Gamma).
\end{align*}
Take $\gamma>0$. Then there exists a measurable function $N_0:\Omega\to \mathbb{N}$ such that for a.e.$x\in\Omega$
and every $t\in[0,1]$ there is some $l\in\{0,1,\cdots,n\}$ such that $f^l(x)\in\Gamma$ and
$\left|(l/n)-t\right|<\gamma$.
\end{lem}

\begin{lem}{\rm\cite{Tho,Tho2}}\label{lem3.5}
(Generalised Pressure Distribution Principle) Let $(M,d,f)$ be a
topological dynamical system. Let $Z\subset M$ be an arbitrary Borel
set. Suppose there exist $\epsilon>0$ and $s\geq0$ such that one can
find a sequence of Borel probability measures $\mu_k,$ a constant
$K>0$ and an integer $N$ satisfying
\begin{align*}
\limsup\limits_{k\to\infty}\mu_k(B_n(x,\epsilon))\leq
K\exp(-ns+\sum\limits_{i=0}^{n-1}\psi(f^ix))
\end{align*}
for every ball $B_n(x,\epsilon)$ such that $B_n(x,\epsilon)\cap
Z\neq\emptyset$ and $n\geq N.$ Furthermore, assume that at least one
limit measure $\nu$ of the sequence $\mu_k$ satisfies $\nu(Z)>0.$
Then $P(Z, \psi,\epsilon)\geq s.$
\end{lem}

\subsubsection{Construction of the Fractal $F$}
Fix $0<\delta<1,\gamma>0$. By (\ref{equ3.1}),
we can choose $\epsilon'$ sufficiently small so
\begin{align}\label{equ3.2}
\text{Var}(\psi,\epsilon'):=\sup\{|\psi(x)-\psi(y)|:d(x,y)\leq\epsilon'\}<\gamma,
\end{align}
and
\begin{align*}
P_\nu^{Kat}(f,\psi,\epsilon')>P_\nu(f,\psi)-\gamma.
\end{align*}
For $\epsilon'>0$ and $\nu\in\mathscr{M}_{inv}(\widetilde{\Lambda}_{\mu},f)$, let us
fix the ingredients obtained by lemma \ref{lem3.3}.

We choose a increasing sequence $l_k\to\infty$ such that
$m_{k,j}(\widetilde{\Lambda}_{l_k})>1-\delta$ for all $1\leq j\leq s_k.$
Let $\eta=\frac{\epsilon'}{4\epsilon_0},$
it follows from weak shadowing lemma that there is a sequence of numbers $\{\delta_k\}.$
Let $\xi_k$ be a finite partition of $M$ with
diam$(\xi_k)<\min\{\frac{b_k(1-\exp(-\epsilon))}{4\sqrt{2}\exp((k+1)\epsilon)},\epsilon_{l_k},\frac{\delta_{l_k}}{3}\}$
and $\xi_k\geq\{\widetilde{\Lambda}_{l_k}, M\setminus \widetilde{\Lambda}_{l_k} \}.$
For $n\in\mathbb{N},$ we consider the set
\begin{align*}
\Lambda^n(m_{k,j})=\{&x\in\widetilde{\Lambda}_{l_k}:f^q(x)\in\xi_k(x){\rm ~for~some~} q\in[n,(1+\gamma)n],\\
&D(\mathscr{E}_n(x),m_{k,j})<\frac{1}{k}
{\rm~and~}\left|\frac{1}{m}S_m\psi(x)-\int \psi dm_{k,j}\right|<\epsilon'{\rm~for~all~} m\geq n\},
\end{align*}
where $\xi_k(x)$ is the element in $\xi_k$ containing $x.$
By Birkhoff ergodic theorem and lemma \ref{lem3.4} we have
$m_{k,j}(\Lambda^n(m_{k,j}))\to m_{k,j}(\widetilde{\Lambda}_{l_k})$ as $n\to\infty$.
So, we can take $n_k\to\infty$ such that
\begin{align*}
m_{k,j}(\Lambda^n(m_{k,j}))>1-\delta
\end{align*}
for all $n\geq n_k$ and $1\leq j\leq s_k.$

For $k\in \mathbb{N}$, let
\begin{align*}
Q(\Lambda^n(m_{k,j}),\epsilon')&
=\inf\bigg\{\sum_{x\in S}\exp\big(\sum_{i=0}^{n-1}\psi(f^ix)\big):S \text{ is }
(n,\epsilon') \text{ spanning set for } \Lambda^n(m_{k,j}) \bigg\},\\
P(\Lambda^n(m_{k,j}),\epsilon')&
=\sup\bigg\{\sum_{x\in S}\exp\big(\sum_{i=0}^{n-1}\psi(f^ix)\big):S \text{ is }
(n,\epsilon') \text{ separated set for } \Lambda^n(m_{k,j})\bigg\}.
\end{align*}
Then for all $n\geq n_k$ and $1\leq j\leq s_k$, we have
\begin{align*}
P(\Lambda^n(m_{k,j}),\epsilon')\geq Q(\Lambda^n(m_{k,j}),\epsilon')\geq N^{m_{k,j}}(\psi,\delta,\epsilon,n).
\end{align*}
We obtain
\begin{align*}
\liminf_{n\to\infty}\frac{1}{n}\log P(\Lambda^n(m_{k,j}),\epsilon')
\geq P_{m_{k,j}}^{Kat}(f,\psi,\epsilon').
\end{align*}
Thus for each $k\in \mathbb{N}$, we can choose $t_k$ large enough such that $\exp(\gamma t_k)>\sharp \xi_k$
and
\begin{align*}
\frac{1}{t_k}\log P(\Lambda^{t_k}(m_{k,j}),\epsilon')>P_{m_{k,j}}^{Kat}(f,\psi,\epsilon')-\gamma,1\leq j\leq s_k.
\end{align*}
Let $S(k,j)$
be a $(t_k,\epsilon')$-separated set for $\Lambda^{t_k}(m_{k,j})$ such that
\begin{align*}
\sum_{x\in S(k,j)}\exp\bigg\{\sum_{i=0}^{t_k-1}\psi(f^ix)\bigg\}
\geq\exp\left(t_k(P_{m_{k,j}}^{Kat}(f,\psi,\epsilon')-2\gamma)\right).
\end{align*}
For each $q\in[t_k,(1+\gamma)t_k],$ let
\begin{align*}
V_q=\{x\in S(k,j):q \text{ is the minimum integer such that } f^q(x)\in\xi_k(x)\}
\end{align*}
and let $n(k,j)$ be the value of $q$ which maximizes
$\sum\limits_{x\in V_q}\exp\left\{\sum\limits_{i=0}^{t_k-1}\psi(f^ix)\right\}$. Obviously,
$n(k,j)\geq t_k$ and $t_k\geq\frac{n(k,j)}{1+\gamma}\geq n(k,j)(1-\gamma).$
Since $\exp({\gamma t_k})\geq\gamma t_k+1,$ we have that
\begin{align*}
     \sum_{x\in V_{n(k,j)}}\exp\bigg\{\sum_{i=0}^{n(k,j)-1}\psi(f^ix)\bigg\}
\geq&\sum_{x\in V_{n(k,j)}}\exp\bigg\{\sum_{i=0}^{t_k-1}\psi(f^ix)\bigg\}\cdot\exp\bigg\{(t_k-n(k,j))\|\psi\|\bigg\}\\
\geq&\frac{1}{\gamma t_k+1}\sum_{x\in S(k,j)}\exp\bigg\{\sum_{i=0}^{t_k-1}\psi(f^ix)\bigg\}\cdot\exp(-n(k,j)\gamma\|\psi\|)\\
\geq&\exp\left(t_k(P_{m_{k,j}}^{Kat}(f,\psi,\epsilon')-3\gamma)-n(k,j)\gamma\|\psi\|\right).
\end{align*}
Consider the element $A_{k,j}\in\xi_k$ such that
$\sum\limits_{x\in V_{n(k,j)}\cap A_{k,j}}\exp\left\{\sum\limits_{i=0}^{n(k,j)-1}\psi(f^ix)\right\}$ is maximal.
Let $W_{n(k,j)}=V_{n(k,j)}\cap A_{k,j}$. It follows that
\begin{align*}
\sum\limits_{x\in W_{n(k,j)}}\exp\bigg\{\sum_{i=0}^{n(k,j)-1}\psi(f^ix)\bigg\}
\geq \frac{1}{\#\xi_k}\exp\left(t_k(P_{m_{k,j}}^{Kat}(f,\psi,\epsilon')-3\gamma)-n(k,j)\gamma\|\psi\|\right).
\end{align*}
Since $\exp(\gamma t_k)>\sharp \xi_k$ and $n(k,j)(1-\gamma)\leq t_k\leq n(k,j)$, we have
\begin{align*}
    &\sum\limits_{x\in W_{n(k,j)}}\exp\bigg\{\sum_{i=0}^{n(k,j)-1}\psi(f^ix)\bigg\}\\
\geq&\exp\left\{n(k,j)(P_{m_{k,j}}^{Kat}(f,\psi,\epsilon')-3\gamma)
      -n(k,j)\gamma|P_{m_{k,j}}^{Kat}(f,\psi,\epsilon')-3\gamma|-n(k,j)\gamma\|\psi\|\right\}\\
\geq&\exp\left\{n(k,j)\left(P_{m_{k,j}}^{Kat}(f,\psi,\epsilon')-3\gamma
      -\gamma H-3\gamma^2-\gamma\|\psi\|\right)\right\}.
\end{align*}
Let
\begin{align*}
R_{k,j}=\sum\limits_{x\in W_{n(k,j)}}\exp\bigg\{\sum_{i=0}^{n(k,j)-1}\psi(f^ix)\bigg\}.
g(\gamma)=3\gamma+\gamma H+3\gamma^2+\gamma\|\psi\|,
\end{align*}
Then $\lim_{\gamma\to0}g(\gamma)=0$ and
\begin{align*}
R_{k,j}\geq \exp\left\{n(k,j)\left(P_{m_{k,j}}^{Kat}(f,\psi,\epsilon')-g(\gamma)\right)\right\}.
\end{align*}

Notice that $A_{n(k,j)}(m_{k,j})$ is contained in an open subset $U(k,j)$ of some Lyapunov neighborhood
with diam$(U(k,j))\leq3\text{diam}(\xi_k).$ By the ergodicity of $\mu,$ for any two
measures $m_{k_1,j_1},m_{k_2,j_2}$ and any natural number $N$, there exists $s=s(k_1,j_1,k_2,j_2)>N$
and $y=y(k_1,j_1,k_2,j_2)\in U(k_1,j_1)\cap\widetilde{\Lambda}_{l_{k_1}}$  such that
$f^s(y)\in U(k_2,j_2)\cap\widetilde{\Lambda}_{l_{k_2}}.$
Letting $C_{k,j}=\frac{a_{k,j}}{n(k,j)},$ we can choose an integer $N_k$
large enough so that $N_kC_{k,j}$ are integers and
\begin{align}\label{equ1}
N_k\geq k\sum\limits_{\substack{1\leq r_1,r_2\leq k+1 \\ 1\leq j_i\leq s_{r_i},i=1,2}}s(r_1,j_1,r_2,j_2).
\end{align}
Let $X_k=\sum\limits_{j=1}^{s_k-1}s(k,j, k,j+1)+s(k,s_k, k,1)$  and
\begin{align}\label{equ2}
Y_k=\sum\limits_{j=1}^{s_k}N_kn(k,j)C_{k,j}+X_k=N_k+X_k,
\end{align}
then we have
\begin{align}
\frac{N_k}{Y_k}\geq \frac{1}{1+\frac{1}{k}}\geq 1-\frac{1}{k}.\label{equ3}
\end{align}
Choose a strictly increasing sequence $\{T_k\}$ with $T_k\in\mathbb{N},$
\begin{align}\label{equ4}
Y_{k+1}\leq\frac{1}{k+1}\sum\limits_{r=1}^kY_rT_r,
\sum\limits_{r=1}^k(Y_rT_r+s(r,1,r+1,1))\leq \frac{1}{k+1}Y_{k+1}T_{k+1}.
\end{align}
For $x\in X,$ we define segments of orbits
\begin{align*}
L_{k,j}(x)&:=(x,f(x),\cdots,f^{n(k,j)-1}(x)), 1\leq j\leq s_k,\\
\widehat{L}_{k_1,j_1,k_2,j_2}(x)&:=(x,f(x),\cdots,f^{s(k_1,j_1,k_2,j_2)-1}(x)),1\leq j_i\leq s_{k_i},i=1,2.
\end{align*}
Consider the pseudo-orbit with finite length
\begin{align*}
O_k=O( &x(1,1,1,1), \cdots,x(1,1,1,N_1C_{1,1}),\cdots,
x(1,s_1,1,1),\cdots,x(1,s_1,1,N_1C_{1,s_1});\\&\cdots;\\
&x(1,1,T_1,1),\cdots,x(1,1,T_1,N_1C_{1,1}),\cdots,x(1,s_1,T_1,1),\cdots, x(1,s_1,T_1,N_1C_{1,s_1});\\
&\vdots\\&
x(k,1,1,1), \cdots,x(k,1,1,N_kC_{k,1}),\cdots,
x(k,s_k,1,1),\cdots,x(k,s_k,1,N_kC_{k,s_k});\\&\cdots;\\&
x(k,1,T_k,1), \cdots,x(k,1,T_k,N_kC_{k,1}),\cdots,
x(k,s_k,T_k,1),\cdots,x(k,s_k,T_k,N_kC_{k,s_k});
)
\end{align*}
with the precise form as follows:
\begin{align*}
\{ &L_{1,1}(x(1,1,1,1)), \cdots,L_{1,1}(x(1,1,1,N_1C_{1,1})),\widehat{L}_{1,1,1,2}(y(1,1,1,2));\\&
L_{1,2}(x(1,2,1,1)), \cdots,L_{1,2}(x(1,2,1,N_1C_{1,2})),\widehat{L}_{1,2,1,3}(y(1,2,1,3));
\cdots,\\&
L_{1,s_1}(x(1,s_1,1,1)), \cdots,L_{1,s_1}(x(1,s_1,1,N_1C_{1,s_1})),\widehat{L}_{1,s_1,1,1}(y(1,s_1,1,1));\\&
\cdots,\\&
L_{1,1}(x(1,1,T_1,1)), \cdots,L_{1,1}(x(1,1,T_1,N_1C_{1,1})),\widehat{L}_{1,1,1,2}(y(1,1,1,2));\\&
L_{1,2}(x(1,2,T_1,1)), \cdots,L_{1,2}(x(1,2,T_1,N_1C_{1,2})),\widehat{L}_{1,2,1,3}(y(1,2,1,3));
\cdots,\\&
L_{1,s_1}(x(1,s_1,T_1,1)), \cdots,L_{1,s_1}(x(1,s_1,T_1,N_1C_{1,s_1})),\widehat{L}_{1,s_1,1,1}(y(1,s_1,1,1));\\&
\widehat{L}(y(1,1,2,1));\\&\vdots,\\&
L_{k,1}(x(k,1,1,1)), \cdots,L_{k,1}(x(k,1,1,N_kC_{k,1})),\widehat{L}_{k,1,k,2}(y(k,1,k,2));\\&
L_{k,2}(x(k,2,1,1)), \cdots,L_{k,2}(x(k,2,1,N_kC_{k,2})),\widehat{L}_{k,2,k,3}(y(k,2,k,3));\cdots\\&
L_{k,s_k}(x(k,s_k,1,1)), \cdots,L_{k,s_k}(x(k,s_k,1,N_kC_{k,s_k})),\widehat{L}_{k,s_k,k,1}(y(k,s_k,k,1));\\
&\cdots\\&
L_{k,1}(x(k,1,T_k,1)), \cdots,L_{k,1}(x(k,1,T_k,N_kC_{k,1})),\widehat{L}_{k,1,k,2}(y(k,1,k,2));\\&
L_{k,2}(x(k,2,T_k,1)), \cdots,L_{k,2}(x(k,2,T_k,N_kC_{k,2})),\widehat{L}_{k,2,k,3}(y(k,2,k,3));\cdots\\&
L_{k,s_k}(x(k,s_k,T_k,1)), \cdots,L_{k,s_k}(x(k,s_k,T_k,N_kC_{k,s_k})),\widehat{L}_{k,s_k,k,1}(y(k,s_k,k,1));\\&
\widehat{L}(y(k,1,k+1,1));
\},
\end{align*}
where $x(q,j,i,t)\in W_{n(q,j)}.$

For $1\leq q\leq k,1\leq i\leq T_q,1\leq j\leq s_q, 1\leq t\leq N_qC_{q,j},$ let $M_1=0,$
\begin{align*}
M_q&=M_{q,1}=\sum\limits_{r=1}^{q-1}(T_rY_r+s(r,1,r+1,1)),\\
M_{q,i}&=M_{q,i,1}=M_q+(i-1)Y_q,\\
M_{q,i,j}&=M_{q,i,j,1}=M_{q,i}+\sum\limits_{p=1}^{j-1}(N_qn(q,p)C_{q,p}+s(k,p,k,p+1)),
\\
M_{q,i,j,t}&=M_{q,i,j}+(t-1)n(q,j).
\end{align*}
By weak shadowing lemma, there exist at least one shadowing point $z$ of $O_k$ such that
\begin{align*}
d(f^{M_{q,i,j,t}+p}(z),f^p(x(q,j,i,t)))\leq \eta\epsilon_0\exp(-\epsilon l_q)\leq\frac{\epsilon'}{4\epsilon_0}\epsilon_0\exp(-\epsilon l_q)\leq \frac{\epsilon'}{4},
\end{align*}
for $1\leq q\leq k, 1\leq i\leq T_q, 1\leq j\leq s_q, 1\leq t\leq N_qC_{q,j}, 1\leq p\leq n(q,j)-1.$
Let $B(x(1,1,1,1),\cdots,x(k,s_k,T_k,N_kC_{k,s_k}))$
be the set of all shadowing points for the above pseudo-orbit. Precisely,
\begin{align*}
B(&x(1,1,1,1),\cdots,x(k,s_k,T_k,N_kC_{k,s_k}))=\\
B(&x(1,1,1,1),\cdots,x(1,1,1,N_1C_{1,1},),\cdots, x(1,s_1,1,1), \cdots, x(1,s_1,1,N_1C_{1,s_1});\\&\cdots;\\&
x(1,1,T_1,1),\cdots,x(1,1,T_1,N_1C_{1,1},),\cdots, x(1,s_1,T_1,1), \cdots, x(1,s_1,T_1,N_1C_{1,s_1});\\&\cdots;\\&
x(k,1,T_1,1),\cdots,x(k,1,1,N_kC_{k,1},),\cdots, x(k,s_k,1,1), \cdots, x(k,s_k,1,N_kC_{k,s_k});\\&\cdots;\\&
x(k,1,T_k,1),\cdots,x(k,1,T_k,N_kC_{k,1},),\cdots, x(k,s_k,T_k,1), \cdots, x(k,s_k,T_k,N_kC_{k,s_k})).
\end{align*}
Then the set $B(x(1,1,1,1),\cdots,x(k,s_k,T_k,N_kC_{k,s_k}))$ can be considered
as a map with variables $x(q,j,i,t)$.
We define $F_k$ by
\begin{align*}
F_k=\bigcup\{B(&x(1,1,1,1),\cdots,x(k,s_k,T_k,N_kC_{k,s_k})):\\
&x(1,1,1,1)\in W_{n(1,1)},\cdots,x(k,s_k,T_k,N_kC_{k,s_k})\in W_{n(k,s_k)}\}.
\end{align*}
Obviously, $F_k$ is non-empty compact and $F_{k+1}\subseteq F_{k}$. Define $F=\bigcap_{k=1}^{\infty}F_k$.

\begin{lem}
$F\subseteq G_{\nu}$.
\begin{proof}
For any $z\in F$, we can choose sufficiently $k$ such that $z\in F_k$.
Assume that $z\in B(x(1,1,1,1),\cdots,x(k,s_k,T_k,N_kC_{k,s_k}))$.
The remaining proof is similar to the proof of lemma 4.4 in \cite{LiaLiaSUnTia}.
\end{proof}
\end{lem}

\subsubsection{Construction of a Special Sequence
of Measures $\alpha_k$}
Now, we construct a sequence of measures to compute the topological entropy of $F$. We first undertake an intermediate
constructions. For each
\begin{align*}
\underline{x}=(x(1,1,1,1),\cdots,x(k,s_k,T_k,N_kC_{k,s_k}))\in W_{n(1,1)}\times\cdots\times W_{n(k,s_k)},
\end{align*}
we choose one point $z=z(\underline{x})$ such that
\begin{align*}
z\in B(x(1,1,1,1),\cdots,x(k,s_k,T_k,N_kC_{k,s_k}))
\end{align*}
Let $L_k$ be the set of all points constructed in this way.
Fix the position indexed
$q,j,i,t,$ for distinct $x(q,j,i,t),x'(q,j,i,t)\in W_{n(q,j)}$,
the corresponding shadowing points $z,z'$ satisfying
\begin{align*}
 &d(f^{M_{q,i,j,t}+p}(z), f^{M_{q,i,j,t}+p}(z'))\\
 \geq &d(f^p(x(q,j,i,t)),f^p(x'(q,j,i,t)))-d(f^{M_{q,i,j,t}+p}(z),f^p(x(q,j,i,t)))\\
       &-d(f^{M_{q,i,j,t}+p}(z'),f^q(x'(q,j,i,t)))\\
 \geq &d(f^q(x(q,j,i,t)),f^p(x'(q,j,i,t)))-\frac{\epsilon'}{2}.
\end{align*}
Noticing that $x(q,j,i,t),x'(q,j,i,t)$ are $(n(q,j),\epsilon')$-separated, we obtain
$f^{M_{q,i,j,t }}(z)$, $f^{M_{q,i,j,t }}(z')$ are $(n(q,j),\epsilon'/2)$-separated.

For each $z\in L_k,$ we
associate a number $\EuScript{L}_k(z)\in (0,\infty).$ Using these
numbers as weights, we define, for each $k,$ an atomic measure
centered on $L_k.$ Precisely, if
\begin{align*}
z\in B(x(1,1,1,1),\cdots,x(k,s_k,T_k,N_kC_{k,s_k})),
\end{align*}
we define
\begin{align*}
\mathcal{L}_k(z)=\mathcal{L}(x(1,1,1,1))\cdots\mathcal{L}(x(k,s_k,T_k,N_kC_{k,s_k})),
\end{align*}
where $\mathcal{L}(x(q,j,i,t))=\exp S_{n(q,j)}\psi(x(q,j,i,t))$.

We define
\begin{align*}
\alpha_k:=\frac{\sum_{z\in L_k}\mathcal{L}_k(z)\delta_z}{\kappa_k},
\end{align*}
where
\begin{align*}
&\kappa_k=\sum_{z\in L_k}\mathcal{L}_k(z)\\
=&\sum_{x(1,1,1,1)\in W_{n(1,1)}}\cdots\sum_{x(k,s_k,T_k,N_kC_{k,s_k})\in W_{n(k,s_k)}}
  \mathcal{L}(x(1,1,1,1))\cdots\mathcal{L}(x(k,s_k,T_k,N_kC_{k,s_k}))\\
=&\left(R_{1,1}^{N_1C_{1,1}}R_{1,2}^{N_1C_{1,2}}\cdots R_{1,s_1}^{N_1C_{1,s_1}}\right)^{T_1}\cdots
\left(R_{k,1}^{N_kC_{k,1}}R_{k,2}^{N_kC_{k,2}}\cdots R_{k,s_k}^{N_kC_{k,s_k}}\right)^{T_k}.
\end{align*}

In order to prove the main results of this paper, we present some lemmas.
\begin{lem}
Suppose $\nu$ is a limit measure of the sequence of probability
measures $\alpha_k.$ Then $\nu(F)=1.$
\end{lem}
\begin{proof}
Suppose $\nu=\lim_{k\to\infty}\alpha_{l_k}$ for $l_k\to\infty$. For any fixed $l$ and all
$p\geq0$, $\alpha_{l+p}(F_l)=1$ since $F_{l+p}\subset F_l$. Thus, $\nu(F_l)\geq\limsup_{k\to\infty}\alpha_{l_k}(F_l)=1$.
It follows that $\nu(F)=\lim_{l\to\infty}\nu(F_l)=1$.
\end{proof}

Let $\EuScript{B}=B_n(x,\frac{\epsilon'}{8})$ be an arbitrary ball which intersets $F$. Let $k$ be an unique number satisfies
$M_{k+1}\leq n<M_{k+2}$. Let $i\in\{1,\cdots,T_{k+1}\}$ be the unique number so
\begin{align*}
M_{k+1,i}\leq n<M_{k+1,i+1}.
\end{align*}
Here we appoint $M_{k+1,T_{k+1}+1}=M_{k+2,1}$. We assume that $i\geq2$, the simpler case $i=1$ is similar.

\begin{lem}\label{lem3.8}
For $p\geq 1$,
\begin{align*}
&\alpha_{k+p}(B_{n}(x,\frac{\epsilon'}{8}))\\
\leq&
\frac{1}
{\kappa_k\left(R_{k+1,1}^{N_{k+1}C_{k+1,1}}R_{k+1,2}^{N_{k+1}C_{k+1,2}}
\cdots R_{k+1,s_{k+1}}^{N_{k+1}C_{k+1,s_{k+1}}}\right)^{i-1}}
\exp\bigg\{S_n\psi(x)+2n\text{Var}(\psi,\epsilon')\\
&+\|\psi\|\bigg(\sum_{q=1}^{k}(T_qX_q+s(q,1,q+1,1))+(i-1)X_{k+1}+Y_{k+1}+s(k+1,1,k+2,1)\bigg)\bigg\}.
\end{align*}
\end{lem}
\begin{proof}
Case $p=1$. Suppose $\alpha_{k+1}(B_{n}(x,\frac{\epsilon'}{8}))>0$,
then $L_{k+1}\cap B_{n}(x,\frac{\epsilon'}{8})
\neq \emptyset$.
Let $z=z(\underline{x},\underline{x}_{k+1}),z^{\prime}=z(\underline{y},\underline{y}_{k+1})
\in L_{k+1}\cap B_{n}(x,\frac{\epsilon'}{8})$, where
\begin{align*}
\underline{x}&=(x(1,1,1,1),\cdots,x(k,s_k,T_k,N_kC_{k,s_k})),\\
\underline{y}&=(y(1,1,1,1),\cdots,y(k,s_k,T_k,N_kC_{k,s_k})),
\end{align*}
and
\begin{align*}
\underline{x}_{k+1}=(&x(k+1,1,1,1),\cdots,x(k+1,s_{k+1},i-1,N_{k+1}C_{k+1,s_{k+1}}),\\
                    &\cdots,x(k+1,s_{k+1},T_k,N_{k+1}C_{k+1,s_{k+1}}))\\
\underline{y}_{k+1}=(&y(k+1,1,1,1),\cdots,y(k+1,s_{k+1},i-1,N_{k+1}C_{k+1,s_{k+1}})\\
                    &\cdots,y(k+1,s_{k+1},T_k,N_{k+1}C_{k+1,s_{k+1}})).
\end{align*}
Since $d_{n}(z,z')<\frac{\epsilon'}{4}$, we have $\underline{x}=\underline{y}$ and
$x(k+1,1,1,1)=y(k+1,1,1,1),\cdots,
x(k+1,s_{k+1},i-1,N_{k+1}C_{k+1,s_{k+1}})=y(k+1,s_{k+1},i-1,N_{k+1}C_{k+1,s_{k+1}})$.
Thus we have
\begin{align*}
\alpha_{k+1}(B_{n}(x,\frac{\epsilon'}{8}))\leq&
\frac{1}{\kappa_{k+1}}
\mathcal{L}(x(1,1,1,1))\cdots\mathcal{L}(x(k+1,s_{k+1},i-1,N_{k+1}C_{k+1,s_{k+1}}))\cdot\\
&\left(R_{k+1,1}^{N_{k+1}C_{k+1,1}} R_{k+1,2}^{N_{k+1}C_{k+1,2}}\cdots R_{k+1,s_{k+1}}^{N_{k+1}C_{k+1,s_{k+1}}}\right)^{T_{k+1}-(i-1)}\\
=&\frac{\mathcal{L}(x(1,1,1,1))\cdots\mathcal{L}(x(k+1,s_{k+1},i-1,N_{k+1}C_{k+1,s_{k+1}}))}
{\kappa_k\left(R_{k+1,1}^{N_{k+1}C_{k+1,1}} R_{k+1,2}^{N_{k+1}C_{k+1,2}}
\cdots R_{k+1,s_{k+1}}^{N_{k+1}C_{k+1,s_{k+1}}}\right)^{i-1}}.
\end{align*}
Since
\begin{align*}
d(f^{M_{q,i,j,t}+p}(z),f^p(x(q,j,i,t)))\leq \frac{\epsilon'}{4},
\end{align*}
$d_n(z,x)<\frac{\epsilon'}{8}$ and $i\geq2$, we have
\begin{align*}
     &\mathcal{L}(x(1,1,1,1))\cdots\mathcal{L}(x(k+1,s_{k+1},i-1,N_{k+1}C_{k+1,s_{k+1}}))\\
\leq &\exp\bigg\{S_n\psi(x)+2n\text{Var}(\psi,\epsilon')\\
&+\|\psi\|\bigg(\sum_{q=1}^{k}(T_qX_q+s(q,1,q+1,1))+(i-1)X_{k+1}+Y_{k+1}+s(k+1,1,k+2,1)\bigg)\bigg\}.
\end{align*}
Case $p>1$ is similar.
\end{proof}

\begin{lem}\label{lem3.9}
For sufficiently large $n$,
\begin{align*}
\limsup\limits_{m\to\infty}\alpha_m(B_n(x,\epsilon))
\leq\exp\left\{-n(P_{\nu}(f,\psi)-5\gamma-g(\gamma))+S_n\psi(x)\right\}.
\end{align*}
\end{lem}
\begin{proof}
Recall that
\begin{align*}
a_{q,j}=n(q,j)C_{q,j},
R_{q,j}\geq \exp\left\{n(q,j)\left(P_{m_{q,j}}^{Kat}(f,\psi,\epsilon')-g(\gamma)\right)\right\}.
\end{align*}
Combing with lemma \ref{lem3.3}, we obtain
\begin{align*}
&\kappa_k\left(R_{k+1,1}^{N_{k+1}C_{k+1,1}} R_{k+1,2}^{N_{k+1}C_{k+1,2}}
\cdots R_{k+1,s_{k+1}}^{N_{k+1}C_{k+1,s_{k+1}}}\right)^{i-1}\\
=&\left(R_{1,1}^{N_1C_{1,1}}R_{1,2}^{N_1C_{1,2}}\cdots R_{1,s_1}^{N_1C_{1,s_1}}\right)^{T_1}\cdots
\left(R_{k,1}^{N_kC_{k,1}}R_{k,2}^{N_kC_{k,2}}\cdots R_{k,s_k}^{N_kC_{k,s_k}}\right)^{T_k}\\
&\left(R_{k+1,1}^{N_{k+1}C_{k+1,1}} R_{k+1,2}^{N_{k+1}C_{k+1,2}}
\cdots R_{k+1,s_{k+1}}^{N_{k+1}C_{k+1,s_{k+1}}}\right)^{i-1}\\
\geq&\exp\bigg\{\sum_{q=1}^{k}\sum_{j=1}^{s_q}T_qN_qC_{q,j}n(q,j)(P_{m_{q,j}}^{Kat}(f,\psi,\epsilon')-g(\gamma))\\
&+\sum_{j=1}^{s_{k+1}}(i-1)N_{k+1}C_{k+1,j}n(k+1,j)(P_{m_{k+1,j}}^{Kat}(f,\psi,\epsilon')-g(\gamma))\bigg\}\\
\geq&\exp\bigg\{\bigg(\sum_{q=1}^{k}T_qN_q
+(i-1)N_{k+1}\bigg)(P_{\nu}^{Kat}(f,\psi,\epsilon')-g(\gamma))\bigg\}\\
\geq&\exp\bigg\{n\bigg((P_{\nu}^{Kat}(f,\psi,\epsilon')-g(\gamma))-
\frac{n-\sum_{q=1}^{k}T_qN_q-(i-1)N_{k+1}}{n}(H-g(\gamma)\bigg)\bigg\}.
\end{align*}
From (\ref{equ2}) and $i\geq2$, we have
\begin{align*}
n-\sum_{q=1}^{k}T_qN_q-(i-1)N_{k+1}
&=n-\sum_{q=1}^{k}T_qY_q-(i-1)Y_{k+1}+\sum_{q=1}^{k}T_qX_q+(i-1)X_{k+1}\\
&\leq Y_{k+1}+\sum_{r=1}^{k+1}s(r,1,r+1,1)+\sum_{q=1}^{k}T_qX_q+(i-1)X_{k+1}.
\end{align*}
By inqualities (\ref{equ1}), (\ref{equ2}), (\ref{equ3}) and (\ref{equ4}) and $i\geq2$, we obtain
\begin{align*}
\lim_{n\to\infty}\frac{n-\sum_{q=1}^{k}T_qN_q-(i-1)N_{k+1}}{n}(H-g(\gamma))=0.
\end{align*}
Thus for sufficiently large $n$, we can deduce that
\begin{align*}
&\kappa_k\left(R_{k+1,1}^{N_{k+1}C_{k+1,1}} R_{k+1,2}^{N_{k+1}C_{k+1,2}}
\cdots R_{k+1,s_{k+1}}^{N_{k+1}C_{k+1,s_{k+1}}}\right)^{i-1}\\
\geq&\exp\bigg\{n\left(P_{\nu}^{Kat}(f,\psi,\epsilon')-g(\gamma)-\gamma\right)\bigg\}\\
\geq&\exp\bigg\{n\left(P_{\nu}(f,\psi)-g(\gamma)-2\gamma\right)\bigg\}.
\end{align*}
By (\ref{equ3.2}) and lemma \ref{lem3.8}, for sufficiently large $n$ and any $p\geq1$,
\begin{align*}
\alpha_{k+p}(B_{n}(x,\frac{\epsilon'}{8}))\leq
&\frac{\exp\left\{S_n\psi(x)+3n\gamma\right\}}
{\kappa_k\left(R_{k+1,1}^{N_{k+1}C_{k+1,1}}R_{k+1,2}^{N_{k+1}C_{k+1,2}}
\cdots R_{k+1,s_{k+1}}^{N_{k+1}C_{k+1,s_{k+1}}}\right)^{i-1}}\\
\leq&\exp\left\{-n(P_{\nu}(f,\psi)-5\gamma-g(\gamma))+S_n\psi(x)\right\}.
\end{align*}
\end{proof}
Applying the Generalised Pressure Distribution Principle, we have
\begin{align*}
P(G_\nu,\psi,\frac{\epsilon'}{8})
\geq P(F,\psi,\frac{\epsilon'}{8})\geq P_{\nu}(f,\psi)-5\gamma-g(\gamma).
\end{align*}
Let $\epsilon'\to0$ and $\gamma\to0$, the proof of theorem \ref{thm2.2} is completed.

\section{Some Applications}
{\bf Example 1 Diffeomorphisms on surfaces}
Let $f:M\to M$ be a $C^{1+\alpha}$ diffeomorphism with $\dim M=2$ and
$h_{top}(f)>0$, then there exists a hyperbolic measure $m\in\mathscr M_{\rm erg}(M,f)$
with Lyapunov exponents $\lambda_1>0>\lambda_2$(see \cite{Pol}). If
$\beta_1=|\lambda_2|$ and $\beta_2=\lambda_1$, then for any $\epsilon>0$
such that $\beta_1,\beta_2>\epsilon$, we have $m(\Lambda(\beta_1,\beta_2,\epsilon))=1$.
Let
\begin{align*}
\widetilde{\Lambda}=\bigcup_{k=1}^{\infty}\text{supp}(m|\Lambda(\beta_1,\beta_2,\epsilon)),
\end{align*}
then for every  $\nu\in \mathscr M_{\rm inv}(\widetilde{\Lambda}_{\mu},f),
\psi\in C^0(M),$ we have
\begin{align*}
P(G_\nu,\psi)=h_\nu(f)+\int \psi d\nu.
\end{align*}

\noindent{\bf Example 2 Nonuniformly hyperbolic systems} In \cite{Kat1}, Katok described a construction
of a diffeomorphism on the 2-torus $\mathbb{T}^2$ with nonzero Lyapunov exponents, which is not an Anosov map.
Let $f_0$ be a linear automorphism given by the matrix
\begin{align*}
  A=\begin{pmatrix}
   2&1\\
   1&1
  \end{pmatrix}
\end{align*}
with eigenvalues $\lambda^{-1}<1<\lambda$.
$f_0$ has a maximal measure $\mu_1$. Let $D_r$ denote the disk of radius $r$ centered
at (0,0), where $r>0$ is small, and put coordinates $(s_1,s_2)$ on $D_r$ corresponding to
the eigendirections of $A$, i.e, $A(s_1,s_2)=(\lambda s_1,\lambda^{-1}s_2)$.
The map $A$ is the time-1 map of the local flow in $D_r$
generated by the following system of differential equations:
\begin{align*}
\frac{ds_1}{dt}=s_1\log\lambda, \frac{ds_2}{dt}=-s_2\log\lambda.
\end{align*}
The Katok map is obtained from $A$ by slowing down these equations near the origin.
It depends upon a real-valued function $\psi$, which is defined on the unit interval $[0,1]$
and has the following properties:
\begin{enumerate}
\item[(1)] $\psi$ is $C^\infty$ except at 0;
\item[(2)] $\psi(0)=0$ and $\psi(u)=1$ for $u\geq r_0$ where $0<r_0<1$;
\item[(3)] $\psi^{\prime}(u)>0$ for every $0<u<r_0$;
\item[(4)] $\int_0^1\frac{du}{\psi(u)}<\infty$.
\end{enumerate}
Fix sufficiently small numbers $r_0<r_1$ and consider the time-1 map $g$ generated
by the following system of differential equations in $D_{r_1}$:
\begin{align*}
\frac{ds_1}{dt}=s_1\psi(s_1^2+s_2^2)\log\lambda, \frac{ds_2}{dt}=-s_2\psi(s_1^2+s_2^2)\log\lambda.
\end{align*}
The map $f$, given as $f(x)=g(x)$ if $x\in D_{r_1}$ and $f(x)=A(x)$ otherwise, defines a homeomorphism of
torus, which is a $C^\infty$ diffeomorphism everywhere except for the origin. To provide the differentiability
of map $f$, the function $\psi$ must satisfy some extra conditions. Namely, the integral $\int_0^1du/\psi$ must
converge ``very slowly" near the origin. We refer the smoothness to \cite{Kat1}. Here $f$ is contained in the
$C^0$ closure of Anosov diffeomorphisms and even more there is a homeomorphism $\pi:\mathbb{T}^2\to \mathbb{T}^2$
such that $\pi\circ f_0=f\circ \pi$. Let $\nu_0=\pi_\ast\mu_1$.

In \cite{LiaLiaSUnTia}, the authors proved
that there exist $0< \epsilon\ll \beta$ and a neighborhood $U$ of $\nu_0$ in
$\mathscr M_{\rm inv}(\mathbb{T}^2,f)$ such that for any ergodic $\nu\in U$ it holds that
$\nu\in \mathscr M_{\rm inv}(\widetilde{\Lambda}(\beta,\beta,\epsilon))$, where
$\widetilde{\Lambda}(\beta,\beta,\epsilon)=\bigcup_{k\geq1}{\rm supp}(\nu_0|\Lambda_k(\beta,\beta,\epsilon))$.

\begin{cro}
For every  $\nu\in \mathscr M_{\rm inv}(\widetilde{\Lambda}(\beta,\beta,\epsilon),f)$,
$\psi\in C^0(\mathbb{T}^2)$, we have $P(G_\nu,\psi)=h_\nu(f)+\int \psi d\nu$.
\end{cro}

In \cite{LiaLiaSUnTia}, the authors also studied the structure of Pesin set $\widetilde{\Lambda}$
for the robustly transitive partially hyperbolic diffeomorphisms described
by Ma\~{n}\'{e} and the robustly transitive non-partially hyperbolic
diffeomorphisms described by Bonatti-Viana.
They showed that for the diffeomorphisms derived from
Anosov systems $\mathscr M_{\rm inv}(\widetilde{\Lambda},f)$
enjoys many members. So our result is applicable to these maps.


\noindent {\bf Acknowledgements.}   The research was supported by
the National Basic Research Program of China
(Grant No. 2013CB834100) and the National Natural Science
Foundation of China (Grant No. 11271191).

\end{document}